\documentclass[12pt, reqno]{amsart}
\usepackage{amsmath, amsthm, amscd, amsfonts, amssymb, graphicx}
\usepackage[bookmarksnumbered, colorlinks, plainpages]{hyperref}
\input{mathrsfs.sty}

\newtheorem{theorem}{Theorem}[section]
\newtheorem{lemma}[theorem]{Lemma}

\newtheorem{corollary}[theorem]{Corollary}
\theoremstyle{definition}

\theoremstyle{remark}

\numberwithin{equation}{section}

\begin{document}
\title{Unitarily invariant norm inequalities for operators}

\author[M. Erfanian Omidvar , M.S. Moslehian, A. Niknam]{M. Erfanian Omidvar $^1$, M. S.
Moslehian$^2$ and A. Niknam$^3$ }
\address{$^1$ Department of Mathematics, Faculty of science,
Islamic Azad University-Mashhad Branch, Mashhad , Iran.}
\email{erfanian@mshdiau.ac.ir}
\address{$^{2, 3}$ Department of Pure Mathematics, Center of Excellence in
Analysis on Algebraic Structures (CEAAS), Ferdowsi University of Mashhad,
P.O. Box 1159, Mashhad 91775, Iran.} \email{moslehian@ferdowsi.um.ac.ir and
moslehian@ams.org} \email{dassamankin@yahoo.co.uk}


\maketitle

\begin{abstract}
We present several operator and norm inequalities for Hilbert space
operators. In particular, we prove that if
$A_{1},A_{2},\ldots,A_{n}\in {\mathbb B}({\mathscr H})$, then
\[|||A_{1}A_{2}^{*}+A_{2}A_{3}^{*}+\cdots+A_{n}A_{1}^{*}|||\leq\left|\left|\left|\sum_{i=1}^{n}A_{i}A_{i}^{*}\right|\right|\right|,\]
for all unitarily invariant norms.\\
We also show that if $A_{1},A_{2},A_{3},A_{4}$ are projections in
${\mathbb B}({\mathscr H})$, then
\begin{eqnarray*}
&&\left|\left|\left|\left(\sum_{i=1}^{4}(-1)^{i+1}A_{i}\right)\oplus0\oplus0\oplus0\right|\right|\right|\nonumber\\&\leq&|||(A_{1}+|A_{3}A_{1}|)\oplus\nonumber
(A_{2}+|A_{4}A_{2}|)\oplus(A_{3}+|A_{1}A_{3}|)\oplus(A_{4}+|A_{2}A_{4}|)|||
\end{eqnarray*}
for any unitarily invariant norm.
\end{abstract}

\section{Introduction and preliminaries}

Let ${\mathbb B}({\mathscr H})$ stand for the $C^{*}$-algebra of all
bounded linear operators on a complex Hilbert space ${\mathscr H}$
with inner product $\langle\cdot, \cdot\rangle$ and let $I$ denote
the identity operator. For $A\in {\mathbb B}({\mathscr H})$, let
$\|A\|=\sup\{\|Ax\|:\|x\|=1\}$ denote the usual operator norm of $A$
and $|A|=(A^{*}A)^{1/2}$ be the absolute value of $A$. For $1\leq p<
\infty$, the Schatten $p$-norm of a compact operator $A$ is defined
by $\|A\|_{p}=({\rm tr}|A|^{p})^{1/p}$, where ${\rm tr}$ is the
usual trace functional. If $A$ and $B$ are operators in ${\mathbb
B}({\mathscr H})$ we use $A\oplus B$ to denote the $2\times2$
operator matrix $\left[\begin{array}{cc}
 A & 0 \\
  0 & B
\end{array}\right]$, regarded as an operator on ${\mathscr H}\oplus {\mathscr H}$. One can show that
\begin{eqnarray}
\label{1.1} \|A\oplus B\|&=&\max(\|A\|,\|B\|)\\
\label{1.2}\|A\oplus
B\|_{p}&=&\left(\|A\|_{p}^{p}+\|B\|_{p}^{p}\right)^{1/p}\end{eqnarray}
An operator $A\in {\mathbb B}({\mathscr H})$ is positive and write
$A\geq0$ if $\langle A(x),x\rangle\geq 0$ for all $x\in {\mathscr
H}$. we say $A\leq$ B whenever $B-A\geq0$.

We consider the wide class of unitarily invariant norms
$|||\cdot|||$. Each of these norms is defined on an ideal in
${\mathbb B}({\mathscr H})$ and it will be implicitly understood
that when we talk of $|||T|||$, then the operator $T$ belongs to the
norm ideal associated with $|||\cdot|||$. Each unitarily invariant
norm $|||\cdot|||$ is characterized by the invariance property
$|||UTV|||=|||T|||$ for all operators $T$ in the norm ideal
associated with $|||\cdot|||$ and for all unitary operators $U$ and
$V$ in ${\mathbb B}({\mathscr H})$. The following are easily follows
from the basic properties of unitarily invariant norms
\begin{eqnarray}
 \label{1.3}|||A\oplus A^{*}|||&=&|||A\oplus A|||,\\
\label{1.4} |||A\oplus
B|||&=&\left|\left|\left|\left[\begin{array}{cc}
  0 & A \\
  B & 0
\end{array}\right]\right|\right|\right|\\
\label{1.5}|||AA^{*}|||&=&|||A^{*}A|||
\end{eqnarray}
for all operators $A,B\in {\mathbb B}({\mathscr H})$. For the general theory of unitarily invariant norms, we refer the reader to Bhatia and Simon \cite{[1], [12]}. \\
It follows from the Fan dominance principle (see \cite{[7]}) that
the following three inequalities for all unitarily invariant norms
are equivalence:
\begin{eqnarray}\label{1.6} |||A|||&\leq&|||B|||,\\
\label{1.7}|||A\oplus0|||&\leq&|||B\oplus0|||,\\
\label{1.8}|||A\oplus A|||&\leq&|||B\oplus B|||.
\end{eqnarray}
It has been shown by Kittaneh \cite{[9]} that if
$A_{1},A_{2},B_{1},B_{2},X,$ and $Y$ are operators in ${\mathbb
B}({\mathscr H})$, then
\[2|||(A_{1}XA_{2}^{*}+B_{1}YB_{2}^{*})\oplus
0|||\]
\begin{eqnarray}\label{1.9}\leq\left|\left|\left|\left[\begin{array}{cc}
 A_{1}^{*}A_{1}X+XA_{2}^{*}A_{2} & A_{1}^{*}B_{1}Y+XA_{2}^{*}B_{2} \\
  B_{1}^{*}A_{1}X+YB_{2}^{*}A_{2} & B_{1}^{*}B_{1}Y+YB_{2}^{*}B_{2}
\end{array}\right]\right|\right|\right|
\end{eqnarray}
for all unitarily invariant norms.\\ It has been shown by Bhatia and
Kittaneh \cite{[2]} that if $A$ and $B$ are operators in ${\mathbb
B}({\mathscr H})$, then
\begin{eqnarray}\label{1.10}|||A^{*}B+B^{*}A|||\leq
|||A^{*}A+B^{*}B|||,\end{eqnarray} for all unitarily invariant
norms. Kittaneh \cite{[10]} proved that if $A$ and $B$ are positive
operators in ${\mathbb B}({\mathscr H})$, then
\begin{eqnarray}\label{1.12}
|||(A+B)\oplus 0|||\leq \left|\left|\left|
\left(A+\left|B^{1/2}A^{1/2}\right|\right) \oplus
\left(B+\left|A^{1/2}B^{1/2}\right|\right)\right|\right|\right|\end{eqnarray}
for any unitarily invariant norm.

\noindent It was shown by Fong \cite{[4]} that if $A\in M_{n}(C)$,
then\begin{eqnarray}\label{1.13}\|AA^{*}-A^{*}A\|\leq\|A\|^{2},\end{eqnarray}
and it was shown by Kittaneh \cite{[8]}
that\begin{eqnarray}\label{1.14}\|AA^{*}+A^{*}A\|\leq\|A^{2}\|+\|A\|^{2}.\end{eqnarray}

In this paper we establish some operator and norm inequalities. We
generalize inequalities \eqref{1.9} and \eqref{1.10} and present a
norm inequality analogue to \eqref{1.12}. Based on our main result,
we provide new proofs of inequalities \eqref{1.13} and \eqref{1.14}.

\section{Main results}
To achieve our main result we need the following lemma.

\begin{lemma}\cite[Theorem 1]{[3]}
If $A$, $B$ and $X$ are operators in ${\mathbb B}({\mathscr H})$,
then
\begin{eqnarray}\label{2.1}2|||AXB^{*}|||\leq |||A^{*}AX+XB^{*}B|||\end{eqnarray}for any unitarily invariant norm.
\end{lemma}

The main result below is an extension of \cite[Theorem 2.2]{[9]}. We
will prove it by an approach different from \cite[Theorem 2.2]{[9]}.

\begin{theorem}\label{th2.2}
Let $A_{i},B_{i},X_{i}\in {\mathbb B}({\mathscr H})$ for
$i,j=1,2\ldots,n$. Then
\begin{eqnarray*}
&&2|||\sum_{i=1}^{n}(A_{i}X_{i}B^{*}_{i})\oplus0\oplus\ldots\oplus0|||\\
&\leq&\left|\left|\left|\left[\begin{array}{cccc}
A_{1}^{*}A_{1}X_{1}+X_{1}B^{*}_{1}B_{1} & A_{1}^{*}A_{2}X_{2}+X_{1}B^{*}_{1}B_{2} & \cdots & A_{1}^{*}A_{n}X_{n}+X_{1}B^{*}_{1}B_{n}\\
A_{2}^{*}A_{1}X_{1}+X_{2}B^{*}_{2}B_{1} & A_{2}^{*}A_{2}X_{2}+X_{2}B^{*}_{2}B_{2} & \cdots & A_{2}^{*}A_{n}X_{n}+X_{2}B^{*}_{2}B_{n}\\
\vdots & \vdots &\ddots & \vdots\\
A_{n}^{*}A_{1}X_{1}+X_{n}B^{*}_{n}B_{1}&
A_{n}^{*}A_{2}X_{2}+X_{n}B^{*}_{n}B_{2} & \cdots
&A_{n}^{*}A_{n}X_{n}+X_{n}B^{*}_{n}B_{n}
\end{array}\right]\right|\right|\right|
\end{eqnarray*}
for all
unitarily invariant norms.
\end{theorem}
\begin{proof}
Consider the following operators on $\oplus_{i=1}^{n}{\mathscr H}$
\[A=\left[\begin{array}{cccc}
A_{1} & A_{2} & \cdots & A_{n}\\
0& 0 & \cdots & 0\\
\vdots & \vdots &\ddots & \vdots\\
0& 0 & \cdots &0
\end{array}\right]\]and\[B=\left[\begin{array}{cccc}
B_{1} & B_{2} & \cdots & B_{n}\\
0& 0 & \cdots & 0\\
\vdots & \vdots &\ddots & \vdots\\
0& 0 & \cdots &0
\end{array}\right]\]and\[X=\left[\begin{array}{cccc}
X_{1} & 0 & \cdots & 0\\
0& X_{2} & \cdots & 0\\
\vdots & \vdots &\ddots & \vdots\\
0& 0 & \cdots &X_{n}
\end{array}\right].\]Then\[2|||\sum_{i=1}^{n}(A_{i}X_{i}B^{*}_{i})\oplus0\oplus\ldots\oplus0|||\\
\]\begin{eqnarray*}&=&2\left|\left|\left|\left[\begin{array}{cccc}
A_{1} & A_{2} & \cdots & A_{n}\\
0& 0 & \cdots & 0\\
\vdots & \vdots &\ddots & \vdots\\
0& 0 & \cdots &0
\end{array}\right]\left[\begin{array}{cccc}
X_{1} & 0 & \cdots & 0\\
0& X_{2} & \cdots & 0\\
\vdots & \vdots &\ddots & \vdots\\
0& 0 & \cdots &X_{n}
\end{array}\right]\left[\begin{array}{cccc}
B_{1}^{*} &0 & \cdots & 0\\
B_{2}^{*}& 0 & \cdots & 0\\
\vdots & \vdots &\ddots & \vdots\\
B_{n}^{*}& 0 & \cdots &0
\end{array}\right]\right|\right|\right|\\&=&2|||AXB^{*}|||\\&\leq&|||A^{*}AX+XB^{*}B|||\qquad\qquad\quad\quad\quad\quad\quad\quad\quad\quad\quad\quad(\rm{by~
inequality~} \eqref{2.1})\\
&=&\left|\left|\left|\left[\begin{array}{cccc}
A_{1}^{*} &0 & \cdots & 0\\
A_{2}^{*}& 0 & \cdots & 0\\
\vdots & \vdots &\ddots & \vdots\\
A_{n}^{*}& 0 & \cdots &0
\end{array}\right]\left[\begin{array}{cccc}
A_{1} & A_{2} & \cdots & A_{n}\\
0& 0 & \cdots & 0\\
\vdots & \vdots &\ddots & \vdots\\
0& 0 & \cdots &0
\end{array}\right]\left[\begin{array}{cccc}
X_{1} & 0 & \cdots & 0\\
0& X_{2} & \cdots & 0\\
\vdots & \vdots &\ddots & \vdots\\
0& 0 & \cdots &X_{n}
\end{array}\right]\right.\right.\right.\\
&&+ \left.\left.\left.\left[\begin{array}{cccc}
X_{1} & 0 & \cdots & 0\\
0& X_{2} & \cdots & 0\\
\vdots & \vdots &\ddots & \vdots\\
0& 0 & \cdots &X_{n}
\end{array}\right]\left[\begin{array}{cccc}
B_{1}^{*} &0 & \cdots & 0\\
B_{2}^{*}& 0 & \cdots & 0\\
\vdots & \vdots &\ddots & \vdots\\
B_{n}^{*}& 0 & \cdots &0
\end{array}\right]\left[\begin{array}{cccc}
B_{1} & B_{2} & \cdots & B_{n}\\
0& 0 & \cdots & 0\\
\vdots & \vdots &\ddots & \vdots\\
0& 0 & \cdots &0
\end{array}\right]\right|\right|\right|\\&=&\left|\left|\left|\left[\begin{array}{cccc}
A_{1}^{*}A_{1}X_{1}+X_{1}B^{*}_{1}B_{1} & A_{1}^{*}A_{2}X_{2}+X_{1}B^{*}_{1}B_{2} & \cdots & A_{1}^{*}A_{n}X_{n}+X_{1}B^{*}_{1}B_{n}\\
A_{2}^{*}A_{1}X_{1}+X_{2}B^{*}_{2}B_{1} & A_{2}^{*}A_{2}X_{2}+X_{2}B^{*}_{2}B_{2} & \cdots & A_{2}^{*}A_{n}X_{n}+X_{2}B^{*}_{2}B_{n}\\
\vdots & \vdots &\ddots & \vdots\\
A_{n}^{*}A_{1}X_{1}+X_{n}B^{*}_{n}B_{1}&
A_{n}^{*}A_{2}X_{2}+X_{n}B^{*}_{n}B_{2} & \cdots
&A_{n}^{*}A_{n}X_{n}+X_{n}B^{*}_{n}B_{n}
\end{array}\right]\right|\right|\right|.\end{eqnarray*}
\end{proof}
\begin{corollary}
Let $A_{1},A_{2},\ldots,A_{n}\in {\mathbb B}({\mathscr H})$. Then
\[\left|\left|\left|A_{1}A_{2}^{*}+A_{2}A_{3}^{*}+\cdots+A_{n}A_{1}^{*}\right|\right|\right|\leq\left|\left|\left|\sum_{i=1}^{n}A_{i}A_{i}^{*}\right|\right|\right|,\]
for all unitarily invariant norms. In particular,
\begin{eqnarray*}\|A_{1}A_{2}^{*}+A_{2}A_{3}^{*}+\cdots+A_{n}A_{1}^{*}\|_{p}\leq \left\|\sum_{i=1}^{n}A_{i}A_{i}^{*}\right\|_{p}\quad{for\  \ 1\leq p\leq\infty}.
\end{eqnarray*}
\end{corollary}
\begin{proof}Letting $B_{i}=A_{i+1}$ for $i=1,2\ldots,n-1$ and
$B_{n}=A_{1}$ and $X_{i}=I$ in Theorem \ref{th2.2} ,  we get
\begin{eqnarray*}
&&2|||A_{1}A_{2}^{*}+A_{2}A_{3}^{*}+\cdots +A_{n}A_{1}^{*}\oplus0\cdots\oplus0|||\\
&\leq&\left|\left|\left|\left[\begin{array}{cccc}
A_{1}^{*}A_{1}+A^{*}_{2}A_{2} & A_{1}^{*}A_{2}+A^{*}_{2}A_{3} & \cdots & A_{1}^{*}A_{n}+A^{*}_{2}A_{1}\\
A_{2}^{*}A_{1}+A^{*}_{3}A_{2} & A_{2}^{*}A_{2}+A^{*}_{3}A_{3} & \cdots & A_{2}^{*}A_{n}+A^{*}_{3}A_{1}\\
\vdots & \vdots &\ddots & \vdots\\
A_{n}^{*}A_{1}+A^{*}_{1}A_{2}& A_{n}^{*}A_{2}+A^{*}_{1}A_{3} &
\cdots &A_{n}^{*}A_{n}+A^{*}_{1}A_{1}
\end{array}\right]\right|\right|\right|\\&\leq&\left|\left|\left|\left[\begin{array}{cccc}
A_{1}^{*}A_{1} & A_{1}^{*}A_{2} & \cdots & A_{1}^{*}A_{n}\\
A_{2}^{*}A_{1}& A_{2}^{*}A_{2} & \cdots & A_{2}^{*}A_{n}\\
\vdots & \vdots &\ddots & \vdots\\
A_{n}^{*}A_{1}& A_{n}^{*}A_{2} & \cdots &A_{n}^{*}A_{n}
\end{array}\right]\right|\right|\right|+\left|\left|\left|\left[\begin{array}{cccc}
A^{*}_{2}A_{2} & A^{*}_{2}A_{3} & \cdots & A^{*}_{2}A_{1}\\
A^{*}_{3}A_{2} & A^{*}_{3}A_{3} & \cdots &A^{*}_{3}A_{1}\\
\vdots & \vdots &\ddots & \vdots\\
A^{*}_{1}A_{2}& A^{*}_{1}A_{3} & \cdots &A^{*}_{1}A_{1}
\end{array}\right]\right|\right|\right|\\&=&\left|\left|\left|\left[\begin{array}{cccc}
A_{1}^{*} & 0 & \cdots & 0\\
A_{2}^{*}& 0 & \cdots & 0\\
\vdots & \vdots &\ddots & \vdots\\
A_{n}^{*}& 0 & \cdots &0
\end{array}\right]\left[\begin{array}{cccc}
A_{1} & A_{2} & \cdots & A_{n}\\
0& 0 & \cdots & 0\\
\vdots & \vdots &\ddots & \vdots\\
0& 0 & \cdots &0
\end{array}\right]\right|\right|\right|\\&&+\left|\left|\left|\left[\begin{array}{cccc}
A_{2}^{*} & 0 & \cdots & 0\\
A_{3}^{*}& 0 & \cdots & 0\\
\vdots & \vdots &\ddots & \vdots\\
A_{1}^{*}& 0 & \cdots &0
\end{array}\right]\left[\begin{array}{cccc}
A_{2} & A_{3} & \cdots & A_{1}\\
0& 0 & \cdots & 0\\
\vdots & \vdots &\ddots & \vdots\\
0& 0 & \cdots &0
\end{array}\right]\right|\right|\right|\\&=&\left|\left|\left|\left[\begin{array}{cccc}
A_{1} & A_{2} & \cdots & A_{n}\\
0& 0 & \cdots & 0\\
\vdots & \vdots &\ddots & \vdots\\
0& 0 & \cdots &0
\end{array}\right]\left[\begin{array}{cccc}
A_{1}^{*} & 0 & \cdots & 0\\
A_{2}^{*}& 0 & \cdots & 0\\
\vdots & \vdots &\ddots & \vdots\\
A_{n}^{*}& 0 & \cdots &0
\end{array}\right]\right|\right|\right|\\&&+\left|\left|\left|\left[\begin{array}{cccc}
A_{2} & A_{3} & \cdots & A_{1}\\
0& 0 & \cdots & 0\\
\vdots & \vdots &\ddots & \vdots\\
0& 0 & \cdots &0
\end{array}\right]\left[\begin{array}{cccc}
A_{2}^{*} & 0 & \cdots & 0\\
A_{3}^{*}& 0 & \cdots & 0\\
\vdots & \vdots &\ddots & \vdots\\
A_{1}^{*}& 0 & \cdots &0
\end{array}\right]\right|\right|\right|\hspace{0.5cm}\mbox{(\rm{by
Property
\eqref{1.5}})}\\&=&2\left|\left|\left|(\sum_{i=1}^{n}A_{i}A_{i}^{*})\oplus0\ldots\oplus0\right|\right|\right|\end{eqnarray*}
By the equivalence of inequalities
 \eqref{1.6} and \eqref{1.7} we have
\[\left|\left|\left|A_{1}A_{2}^{*}+A_{2}A_{3}^{*}+\cdots+A_{n}A_{1}^{*}\right|\right|\right|\leq\left|\left|\left|\sum_{i=1}^{n}A_{i}A_{i}^{*}\right|\right|\right|\]
for all unitarily invariant norms.
\end{proof}
To establish the next result we need the following Lemma. The lemma
is a basic triangle inequality comparing, in unitarily invariant
norms, the sum of two normal operators to the sum of their absolute
values.
\begin{lemma}\cite{[5]} If $A$ and $B$ are normal operators in ${\mathbb B}({\mathscr H})$, then
\begin{eqnarray}\label{2.2}
|||A+B|||\leq|||\ |A|+|B|\ |||.
\end{eqnarray}
for all unitarily invariant norms.\end{lemma}
\begin{corollary}
Let $A_{1},A_{2},A_{3},A_{4}$  be projections in ${\mathbb
B}({\mathscr H})$. Then \begin{eqnarray}\label{2.3}
&&\left|\left|\left|\left(\sum_{i=1}^{4}(-1)^{i+1}A_{i}\right)\oplus0\oplus0\oplus0\right|\right|\right|\nonumber\\&\leq&|||(A_{1}+|A_{3}A_{1}|)\oplus\nonumber
(A_{2}+|A_{4}A_{2}|)\oplus(A_{3}+|A_{1}A_{3}|)\oplus(A_{4}+|A_{2}A_{4}|)|||\nonumber\\
\end{eqnarray} for all unitarily invariant norms. In particular,
\begin{eqnarray*}&&\left\|\sum_{i=1}^{4}(-1)^{i+1}A_{i}\right\|\\&\leq&\max\{\|A_{1}+|A_{3}A_{1}|\|,
\|A_{2}+|A_{4}A_{2}|\|,\|A_{3}+|A_{1}A_{3}|\|,\|A_{4}+|A_{2}A_{4}|\|\}\end{eqnarray*}
and
\begin{eqnarray*}&&\left\|\sum_{i=1}^{4}(-1)^{i+1}A_{i}\right\|_{p}\\
&\leq&\left(\left\|A_{1}+|A_{3}A_{1}|\right\|_{p}^{p}+\left\|A_{2}+|A_{4}A_{2}|\right\|_{p}^{p}+\left\|A_{3}
+|A_{1}A_{3}|\right\|_{p}^{p}+\left\|A_{4}+|A_{2}A_{4}|\right\|_{p}^{p}\right)^{1/p}\\
&&\qquad\qquad\qquad\qquad\qquad\qquad\qquad\qquad\qquad\qquad\qquad\qquad{\rm
~} (1\leq p<\infty).
\end{eqnarray*}
\end{corollary}
\begin{proof} Letting $n=4$, and replacing $ A_{i}$ and $B_{i}$ by $A_{i}$,
 and $X_{i}=(-1)^{i+1}I$ for $i=1,2,3,4$ in Theorem
 \ref{th2.2}, we get
\begin{eqnarray*}
&&\left|\left|\left|\left[
\begin{array}{cccc}
  \sum_{i=1}^{4}(-1)^{i+1}A_{i} & 0 & 0 & 0 \\
  0 & 0 & 0 & 0 \\
  0 & 0 & 0 & 0 \\
  0 & 0 & 0 & 0 \\
\end{array}\right]\right|\right|\right|\\
  &\leq&\left|\left|\left|\left[\begin{array}{cccc}
  A_{1} & 0 & A_{1}A_{3} & 0 \\
  0 & -A_{2} & 0 & -A_{2}A_{4} \\
  A_{3}A_{1} & 0 & A_{3} & 0 \\
  0 & -A_{4}A_{2} & 0 & -A_{4} \\
\end{array}\right]\right|\right|\right|\\&=&\left|\left|\left|\left[\begin{array}{cccc}
  A_{1} & 0 & A_{1}A_{3} & 0 \\
  0 & A_{2} & 0 & A_{2}A_{4} \\
  A_{3}A_{1} & 0 & A_{3} & 0 \\
  0 & A_{4}A_{2} & 0 & A_{4} \\
\end{array}\right]\right|\right|\right|\quad\mbox{(\rm{~by~unitary~invariance~of~~the~norm~})}\\&=&\left|\left|\left|\left[\begin{array}{cccc}
  A_{1} & 0 & 0 & 0 \\
  0 & A_{2} & 0 & 0 \\
  0 & 0 & A_{3} & 0 \\
  0 & 0 & 0 & A_{4} \\
\end{array}\right]
+\left[\begin{array}{cccc}
  0 & 0 & A_{1}A_{3} & 0 \\
  0 & 0 & 0 & A_{2}A_{4} \\
  A_{3}A_{1} & 0 & 0 & 0 \\
  0 & A_{4}A_{2} & 0 & 0 \\
\end{array}\right]\right|\right|\right|\\&\leq&\left|\left|\left|\ \left|\left[\begin{array}{cccc}
  A_{1} & 0 & 0 & 0 \\
  0 & A_{2} & 0 & 0 \\
  0 & 0 & A_{3} & 0 \\
  0 & 0 & 0 & A_{4} \\
\end{array}\right]\right|
+\left|\left[\begin{array}{cccc}
  0 & 0 & A_{1}A_{3} & 0 \\
  0 & 0 & 0 & A_{2}A_{4} \\
  A_{3}A_{1} & 0 & 0 & 0 \\
  0 & A_{4}A_{2} & 0 & 0 \\
\end{array}\right]\right|\ \right|\right|\right|\quad\mbox{(\rm{by \eqref{2.2}})}\\&=&\left|\left|\left|\left[\begin{array}{cccc}
  A_{1} & 0 & 0 & 0 \\
  0 & A_{2} & 0 & 0 \\
  0 & 0 & A_{3} & 0 \\
  0 & 0 & 0 & A_{4} \\
\end{array}\right]
+\left[\begin{array}{cccc}
  |A_{3}A_{1}| & 0 & 0 & 0 \\
  0 & |A_{4}A_{2}| & 0 & 0 \\
  0 & 0 & |A_{1}A_{3}| & 0 \\
  0 & 0 & 0 & |A_{2}A_{4}| \\
\end{array}\right]\right|\right|\right|\\&=&\left|\left|\left|\left[\begin{array}{cccc}
  A_{1}+|A_{3}A_{1}| & 0 & 0 & 0 \\
  0 & A_{2}+|A_{4}A_{2}| & 0 & 0 \\
  0 & 0 & A_{3}+|A_{1}A_{3}| & 0 \\
  0 & 0 & 0 & A_{4}+|A_{2}A_{4}| \\
\end{array}\right]\right|\right|\right| .\end{eqnarray*} This prove inequality \eqref{2.3}.\\The rest inequalities follow
from \eqref{2.3}, \eqref{1.1} and \eqref{1.2}.
\end{proof}
\begin{corollary}
Let $A_{1},A_{2},\ldots,A_{n}$ be positive operators in ${\mathbb
B}({\mathscr H})$. Then
\begin{eqnarray}\label{2.4}\|A_{1}+A_{2}+\cdots+A_{n}\|\leq
\max\big\{\|A_{i}+(n-1)\|A_{i}\|\|: i=1,2,\ldots,n \big\}.\end{eqnarray}
\end{corollary}
\begin{proof}
First we show that
\begin{eqnarray}\label{2.5} &&\left[\begin{array}{cccc}
A_{1} &A_{1}^{1/2}A_{2}^{1/2}& \cdots & A_{1}^{1/2}A_{n}^{1/2}\\
A_{2}^{1/2}A_{1}^{1/2}& A_{2} &  \cdots &A_{2}^{1/2}A_{n}^{1/2}\\
\vdots & \vdots &\ddots & \vdots\\
A_{n}^{1/2}A_{1}^{1/2} & A_{n}^{1/2}A_{2}^{1/2} & \cdots &A_{n}
\end{array}\right]\\&\leq &\left[\begin{array}{cccc}
A_{1}+(n-1)\|A_{1}\| &0& \cdots & 0\\
0& A_{2}+(n-1)\|A_{2}\| &  \cdots &0\\
\vdots & \vdots &\ddots & \vdots\\
0 & 0 & \cdots &A_{n}+(n-1)\|A_{n}\| \end{array}\right].\nonumber
\end{eqnarray}
It is enough to show that
\[C=\left[\begin{array}{cccc}
(n-1)A_{1} &-A_{1}^{1/2}A_{2}^{1/2}& \cdots & -A_{1}^{1/2}A_{n}^{1/2}\\
-A_{2}^{1/2}A_{1}^{1/2}&(n-1)A_{2} &  \cdots &-A_{2}^{1/2}A_{n}^{1/2}\\
\vdots & \vdots &\ddots & \vdots\\
-A_{n}^{1/2}A_{1}^{1/2} &- A_{n}^{1/2}A_{2}^{1/2} & \cdots
&(n-1)A_{n}
\end{array}\right]\geq0.\]
To see this, we note that
\begin{eqnarray*}
nC=\left[\begin{array}{cccc}
(n-1)A_{1}^{1/2} &-A_{1}^{1/2}& \cdots & -A_{1}^{1/2}\\
-A_{2}^{1/2}&(n-1)A_{2}^{1/2} &  \cdots &-A_{2}^{1/2}\\
\vdots & \vdots &\ddots & \vdots\\
-A_{n}^{1/2} &- A_{n}^{1/2} & \cdots &(n-1)A_{n}^{1/2}
\end{array}\right] \times\\
\left[\begin{array}{cccc}
(n-1)A_{1}^{1/2} &-A_{2}^{1/2}& \cdots & -A_{n}^{1/2}\\
-A_{1}^{1/2}&(n-1)A_{2}^{1/2} &  \cdots &-A_{n}^{1/2}\\
\vdots & \vdots &\ddots & \vdots\\
-A_{1}^{1/2} &- A_{2}^{1/2} & \cdots &(n-1)A_{n}^{1/2}
\end{array}\right]\geq 0
\end{eqnarray*}
Next, by letting $X_{i}=I$ and
replacing both $A_{i}$ and $B_{i}$ by $A_{i}^{1/2}$ in Theorem
\ref{th2.2}, we obtain
\begin{eqnarray*}\|A_{1}+A_{2}+\cdots+A_{n}\|&\leq&\left\|\left[\begin{array}{cccc}
A_{1} &A_{1}^{1/2}A_{2}^{1/2}& \cdots & A_{1}^{1/2}A_{n}^{1/2}\\
A_{2}^{1/2}A_{1}^{1/2}& A_{2} &  \cdots &A_{2}^{1/2}A_{n}^{1/2}\\
\vdots & \vdots &\ddots & \vdots\\
A_{n}^{1/2}A_{1}^{1/2} & A_{n}^{1/2}A_{2}^{1/2} & \cdots &A_{n}
\end{array}\right]\right\|\end{eqnarray*}
\begin{eqnarray*}&\leq& \left\|\left[\begin{array}{cccc}
A_{1}+(n-1)\|A_{1}\| &0& \cdots & 0\\
0& A_{2}+(n-1)\|A_{2}\| &  \cdots &0\\
\vdots & \vdots &\ddots & \vdots\\
0 & 0 & \cdots &A_{n}+(n-1)\|A_{n}\| \end{array}\right]\right\|\\
&&\qquad\qquad\qquad\qquad\qquad\qquad\qquad\qquad\qquad\qquad\qquad({\rm by~} \eqref{2.5}).
\end{eqnarray*}
Hence  \[\|A_{1}+A_{2}+\cdots+A_{n}\|\leq
\max\big\{\big\|A_{i}+(n-1)\|A_{i}\|\big\|: i=1,2,\ldots,n \big\}.\]
\end{proof}
\begin{corollary}
Let $A$ and $B$ be normal operators in ${\mathbb B}(\mathscr H)$, then
\[\|A+B\|\leq\max\big\{\big\||A|+\|A\|\big\|,\big\||B|+\|B\|\big\|\big\}.\]
\end{corollary}
\begin{proof}
Letting  $n=2$ , $A_{1}=|A|$ , $A_{2}=|B|$ in \eqref{2.4}, therefore
\begin{eqnarray*}
\|A+B\|&\leq&\big \||A|+|B|\big\|\qquad\qquad\qquad({\rm by~} \eqref{2.2})\\
&\leq&\max\big\{\big\||A|+\|A\|\big\|,\big\||B|+\|B\|\big\|\big\}.
\end{eqnarray*}
\end{proof}
To establish the next result we need the following lemma.

\begin{lemma}\cite[Theorem 1.1]{[6]} If $A,B,C$ and $D$ are operators in ${\mathbb B}({\mathscr H})$, then
\begin{eqnarray}\label{2.6}
\left\|\left[\begin{array}{cc}
A & B \\
C & D
\end{array}\right]\right\|
\leq\left\|\left[\begin{array}{cc}
\|A\| & \|B\| \\
\|C\| & \|D\|
\end{array}\right]\right\|.
\end{eqnarray}\end{lemma}

\begin{corollary}
Let $A\in {\mathbb B}({\mathscr H})$. Then
\begin{eqnarray}\label{2.7}
\|AA^{*}+A^{*}A\|&\leq&\|A^{2}\|+\|A\|^{2}
\end{eqnarray}
and
\begin{eqnarray}\label{2.8}
\|AA^{*}-A^{*}A\|&\leq&\|A\|^{2}.\end{eqnarray}
\end{corollary}
\begin{proof}
Letting  $n=2$ , $A_{1}=B_{1}=A$ , $A_{2}=B_{2}=A^{*}$ and
$X_1=X_2=I$ in Theorem \ref{th2.2} to get
\begin{eqnarray*}
\|AA^{*}+A^{*}A\| &\leq&\left\|\left[\begin{array}{cc}
  A^{*}A & A^{* 2} \\
  A^{2} & AA^{*}
\end{array}\right]\right\|\\
&\leq&\left\|\left[\begin{array}{cc}
  \|A^{*}A\| & \|A^{* 2}\| \\
  \|A^{2}\| & \|AA^{*}\|
\end{array}\right]\right\|\hspace{2cm}\mbox{(\rm{by
inequality \eqref{2.2}})}
\end{eqnarray*}
Since
\[\left[\begin{array}{cc}
\|A^{*}A\| & \|A^{* 2}\| \\
\|A^{2}\| & \|AA^{*}\|
\end{array}\right]=\left[\begin{array}{cc}
\|A\|^{2} & \|A^{2}\| \\
\|A^{2}\| & \|A\|^{2}
\end{array}\right]\]is self-adjoint, the usual operator norm of this matrix is equal to its spectral radius,
so
 \[\left\|\left[\begin{array}{cc}
  \|A\|^{2} & \|A^{2}\| \\
  \|A^{2}\| & \|A\|^{2}
  \end{array}\right]\right\|=\|A^{2}\|+\|A\|^{2}.\]

This prove inequality \eqref{2.7}. To prove inequality \eqref{2.8},
letting $n=2$ , $A_{1}=B_{1}=A$ , $A_{2}=B_{2}=A^{*}$ and
$X_1=I=-X_2$ in Theorem \ref{th2.2}, we get
\begin{eqnarray*}\|AA^{*}-A^{*}A\|
&\leq&\left\|\left[\begin{array}{cc}
  A^{*}A & 0 \\
  0 & AA^{*}
\end{array}\right]\right\|\\
&\leq&\left\|\left[\begin{array}{cc}
  \|A^{*}A\| & 0 \\
 0 & \|AA^{*}\|
\end{array}\right]\right\|\hspace{2cm}\mbox{(\rm{by
inequality
\eqref{2.2}})}\\\\
&=&\|A\|^{2}.
\end{eqnarray*}
\end{proof}

\end{document}